\documentclass[12pt]{article}
\usepackage{amscd,amsmath,amssymb,amsfonts,braket,mathrsfs,latexsym,enumerate,amsthm}
\usepackage[all]{xy}
\makeatletter
\@namedef{subjclassname@2010}{%
  \textup{2010} Mathematics Subject Classification}
\makeatother

\newtheorem{thm}{Theorem} 
\newtheorem{lem}[thm]{Lemma} 
\newtheorem{cor}[thm]{Corollary}
\newtheorem{prop}[thm]{Proposition}

\theoremstyle{definition}
\newtheorem{rmk}[thm]{Remark}

\numberwithin{thm}{section}
\numberwithin{equation}{section}

\newcommand{\C}{\mathbb{C}}
\newcommand{\F}{\mathbb{F}}

\newcommand{\Q}{\mathbb{Q}}

\newcommand{\Kb}{\overline{K}}

\newcommand{\Kab}{K^{\text{ab}}}
\newcommand{\Gab}{\mathrm{G}^{\text{ab}}}

\newcommand{\R}{\mathbb{R}}
\newcommand{\Z}{\mathbb{Z}}

\newcommand{\alphab}{\overline{\alpha}}

\newcommand{\mfq}{\mathfrak{q}}
\newcommand{\mfI}{\mathfrak{I}}
\newcommand{\mfL}{\mathfrak{L}}

\newcommand{\mfM}{\mathfrak{M}}
\newcommand{\mfP}{\mathfrak{P}}
\newcommand{\mfQ}{\mathfrak{Q}}

\newcommand{\util}{\widetilde{u}}

\newcommand{\cB}{\mathcal{B}}
\newcommand{\cC}{\mathcal{C}}
\newcommand{\cD}{\mathcal{D}}

\newcommand{\cP}{\mathcal{P}}

\newcommand{\cO}{\mathcal{O}}

\newcommand{\cS}{\mathcal{S}}

\newcommand{\Gal}{\mathrm{Gal}}
\newcommand{\G}{\mathrm{G}}
\newcommand{\M}{\mathrm{M}}
\newcommand{\N}{\mathrm{N}}

\newcommand{\End}{\mathrm{End}}
\newcommand{\Aut}{\mathrm{Aut}}

\newcommand{\Norm}{\mathrm{Norm}}

\newcommand{\Nrd}{\mathrm{Nrd}}

\newcommand{\cf}{cf.\ }
\newcommand{\inj}{\hookrightarrow}

\newcommand{\resp}{resp.\ }

\begin{document}

\title{Non-existence of points rational over number fields on Shimura curves}
\author{Keisuke Arai}
\date{}

\pagestyle{plain}

\maketitle

\begin{abstract}

Jordan, Rotger and de Vera-Piquero proved that Shimura curves have
no points rational over imaginary quadratic fields under a certain
assumption.
In this article, we expand their results to the case of number fields
of higher degree.
We also give counterexamples to the Hasse principle on Shimura curves.

\end{abstract}

\section{Introduction}
\label{intro}

Let $B$ be an indefinite quaternion division algebra over $\Q$,
and $d(B)$ its discriminant.
Fix a maximal order $\cO$ of $B$.
A \textit{QM-abelian surface by $\cO$} over a field $F$ is a pair $(A,i)$ where
$A$ is a $2$-dimensional abelian variety over $F$, and 
$i:\cO\inj\End_F(A)$ 
is an injective ring homomorphism satisfying $i(1)=id$
(\cf \cite[p.591]{Bu}).
Here, $\End_F(A)$ is the ring of endomorphisms of $A$ defined over $F$.
We assume that $A$ has a left $\cO$-action.
Let $M^B$ be the Shimura curve over $\Q$ associated to $B$,
which parameterizes isomorphism classes
of QM-abelian surfaces by $\cO$ (\cf \cite[p.93]{J}).
We know that $M^B$ is a proper smooth curve over $\Q$.
For an imaginary quadratic field $k$, we have $M^B(k)=\emptyset$
under a certain assumption (\cite[Theorem 6.3]{J}, \cite[Theorem 1.1]{RdV}).
We expand this result to the case of number fields
of higher degree in this article.
The method of the proof is based on the strategy in \cite{J},
and the key is to control the field of definition of the QM-abelian surface
corresponding to a rational point on $M^B$.
We also give counterexamples to the Hasse principle on $M^B$ over number fields.
We will discuss the relevance to the Manin obstruction in
a forthcoming article.

For a prime number $q$,
let $\cB(q)$ be the set of 
isomorphism classes of
indefinite quaternion division
algebras $B$ over $\Q$ such that
\begin{equation*}
\begin{cases}
B\otimes_{\Q}\Q(\sqrt{-q})\not\cong\M_2(\Q(\sqrt{-q})) & \text{if $q\ne 2$},\\
B\otimes_{\Q}\Q(\sqrt{-1})\not\cong\M_2(\Q(\sqrt{-1}))\ 
\text{and}\ 
B\otimes_{\Q}\Q(\sqrt{-2})\not\cong\M_2(\Q(\sqrt{-2}))
 & \text{if $q=2$}.
 \end{cases}
 \end{equation*}
For positive integers $N$ and $e$, let
$$\cC(N,e):=\Set{\alpha^e+\overline{\alpha}^e\in\Z|
\text{$\alpha\in\C$ is a root of $T^2+sT+N$
for some $s\in\Z$, $s^2\leq 4N$}
},$$
$$\cD(N,e):=\Set{a, a\pm N^{\frac{e}{2}}, a\pm 2N^{\frac{e}{2}},
a^2-3N^e \in\R | a\in\cC(N,e)}.$$
Here, $\overline{\alpha}$ is the complex conjugate of $\alpha$.
If $e$ is even, then $\cD(N,e)\subseteq\Z$.
For a subset $\cD\subseteq\Z$, let
$$\cP(\cD):=\Set{\text{prime divisors of some of the integers
in $\cD\setminus\{0\}$}}.$$
%
For a number field $k$ and a prime $\mfq$ of $k$ of residue characteristic $q$, let
\begin{itemize}
\item
$\kappa(\mfq)$: the residue field of $\mfq$,
\item
$\N_{\mfq}$: the cardinality of $\kappa(\mfq)$,
\item
$e_{\mfq}$: the ramification index of $\mfq$ in $k/\Q$,
\item
$f_{\mfq}$: the degree of the extension $\kappa(\mfq)/\F_q$,
\item
$\cS(k,\mfq)$: the set of isomorphism classes of
indefinite quaternion division
algebras $B$ over $\Q$ such that any prime divisor of $d(B)$ belongs to
\begin{equation*}
\begin{cases}
\cP(\cD(\N_{\mfq},e_{\mfq}))\cup\{q\} & 
\text{if $B\otimes_{\Q}k\cong\M_2(k)$ and $e_{\mfq}$ is even},\\
\cP(\cD(\N_{\mfq},2e_{\mfq}))\cup\{q\} & 
\text{if $B\otimes_{\Q}k\not\cong\M_2(k)$.}
\end{cases}
\end{equation*}
\end{itemize}
Note that $\cS(k,\mfq)$ is a finite set.
The main result of this article is:

\begin{thm}
\label{mainthm}

Let $k$ be a number field of even degree, and
$q$ a prime number such that
\begin{itemize}
\item
there is a unique prime $\mfq$ of $k$ above $q$,
\item
$f_{\mfq}$ is odd (and so $e_{\mfq}$ is even), and
\item
$B\in\cB(q)\setminus\cS(k,\mfq)$.
\end{itemize}
Then
$M^B(k)=\emptyset$.

\end{thm}

\begin{rmk}
\label{MBR}
\begin{enumerate}[\upshape (1)]
\item
By \cite[Theorem 0]{Sh}, we have $M^B(\R)=\emptyset$.
\item
If $k$ is of odd degree, then $k$ has a real place, and so
$M^B(k)=\emptyset$.
\end{enumerate}
\end{rmk}

\section{Canonical isogeny characters}
\label{isogeny}

In this section, we review canonical isogeny characters associated to QM-abelian surfaces,
which were introduced in \cite[\S 4]{J}.
Let $K$ be a number field,
$\Kb$ an algebraic closure of $K$,
$\G_K=\Gal(\Kb/K)$ the absolute Galois group of $K$,
$\cO_K$ the ring of integers of $K$,
$(A,i)$ a QM-abelian surface by $\cO$
over $K$, and
$p$ a prime divisor of $d(B)$.
Then the $p$-torsion subgroup $A[p](\Kb)$ of $A$ has exactly one non-zero
proper left $\cO$-submodule, which we shall denote by $C_p$.
Then $C_p$ has order $p^2$, and is stable under the action of $\G_K$.
Let $\mfP_{\cO}\subseteq\cO$ be the unique left ideal of reduced norm $p\Z$.
In fact, $\mfP_{\cO}$ is a two-sided ideal of $\cO$.
Then $C_p$ is free of rank $1$ over $\cO/\mfP_{\cO}$.
Fix an isomorphism $\cO/\mfP_{\cO}\cong\F_{p^2}$.
The action of $\G_K$ on $C_p$ yields a character
\begin{equation*}
\label{canisog}
\varrho_p:\G_K\longrightarrow\Aut_{\cO}(C_p)\cong\F_{p^2}^{\times}.
\end{equation*}
Here, $\Aut_{\cO}(C_p)$ is the group of 
$\cO$-linear automorphisms of $C_p$.
The character $\varrho_p$ depends on the choice of the isomorphism
$\cO/\mfP_{\cO}\cong\F_{p^2}$, but the pair $\{\varrho_p,(\varrho_p)^p\}$
is independent of this choice.
Either of the characters $\varrho_p,(\varrho_p)^p$ is called a
\textit{canonical isogeny character} at $p$.
We have an induced character
$$\varrho^{\text{ab}}_p:\Gab_K\longrightarrow\F_{p^2}^{\times},$$
where $\Gab_K$ is the Galois group of the maximal abelian extension $\Kab/K$.

For a prime $\mfL$ of $K$,
let $\cO_{K,\mfL}$ be the completion of $\cO_K$ at $\mfL$, and
$$r_p(\mfL):\cO_{K,\mfL}^{\times}\longrightarrow\F_{p^2}^{\times}$$
the composition
\begin{equation*}
\begin{CD}
\cO_{K,\mfL}^{\times}@>\text{$\omega_{\mfL}$}>>
\Gab_K@>\text{$\varrho^{\text{ab}}_p$}>>\F_{p^2}^{\times}.
\end{CD}
\end{equation*}
Here $\omega_{\mfL}$ is the Artin map.

\begin{prop}[{\cite[Proposition 4.7 (2)]{J}}]
\label{r12}

If $\mfL\nmid p$, then $r_p(\mfL)^{12}=1$.

\end{prop}

Fix a prime $\mfP$ of $K$ above $p$.
Then we have an isomorphism $\kappa(\mfP)\cong\F_{p^{f_{\mfP}}}$
of finite fields.
Let $t_{\mfP}:=\text{gcd}(2,f_{\mfP})\in\{1,2\}$.

\begin{prop}[{\cite[Proposition 4.8]{J}}]
\label{r&c}
\begin{enumerate}[\upshape (1)]
\item
There is a unique element $c_{\mfP}\in\Z/(p^{t_{\mfP}}-1)\Z$ satisfying
$r_p(\mfP)(u)=\Norm_{\kappa(\mfP)/\F_{p^{t_{\mfP}}}}(\util)^{-c_{\mfP}}$
for any $u\in\cO_{K,\mfP}^{\times}$.
Here, $\util\in\kappa(\mfP)$ is the reduction of $u$ modulo $\mfP$.
\item
$\displaystyle\frac{2c_{\mfP}}{t_{\mfP}}\equiv e_{\mfP}\bmod{(p-1)}$.
\end{enumerate}
\end{prop}

\begin{cor}
\label{rl2}

For any prime number $l\ne p$, we have
$r_p(\mfP)(l^{-1})^2=l^{e_{\mfP} f_{\mfP}}\bmod{p}$.

\end{cor}

\begin{proof}

$r_p(\mfP)(l^{-1})^2
=(\Norm_{\kappa(\mfP)/\F_{p^{t_{\mfP}}}}(l^{-1})^{-c_{\mfP}})^2
=\Norm_{\F_{p^{f_{\mfP}}}/\F_{p^{t_{\mfP}}}}(l)^{2c_{\mfP}}
=l^{\frac{2c_{\mfP}f_{\mfP}}{t_{\mfP}}}
=l^{e_{\mfP} f_{\mfP}}\bmod{p}$.

\end{proof}

For a prime number $l$, the action of $\G_K$ on the $l$-adic Tate module $T_lA$
yields a representation
$$R_l:\G_K\longrightarrow\Aut_{\cO}(T_lA)\cong\cO_l^{\times}\subseteq B_l^{\times},$$
where $\Aut_{\cO}(T_lA)$ is the group of automorphisms of $T_lA$
commuting with the action of $\cO$,
and $\cO_l=\cO\otimes_{\Z}\Z_l$, $B_l=B\otimes_{\Q}\Q_l$.
Let $\Nrd_{B_l/\Q_l}$ be the reduced norm on $B_l$.
Let $\mfM$ be a prime of $K$, and
$F_{\mfM}\in\G_K$ a Frobenius element at $\mfM$.
For each $e\geq 1$, there is an integer $a(F_{\mfM}^e)\in\Z$ satisfying
$$\Nrd_{B_l/\Q_l}(T-R_l(F_{\mfM}^e))=T^2-a(F_{\mfM}^e)T+(\N_{\mfM})^e\in\Z[T]$$
for any $l$ prime to $\mfM$.

\begin{prop}[{\cite[Proposition 5.3]{J}}]
\label{a&rho}
\begin{enumerate}[\upshape (1)]
\item
We have $a(F_{\mfM}^e)^2\leq 4(\N_{\mfM})^e$
for any positive integer $e$.
\item
Assume $\mfM\nmid p$.
Then
$$a(F_{\mfM}^e)\equiv\varrho_p(F_{\mfM}^e)+(\N_{\mfM})^e \varrho_p(F_{\mfM}^e)^{-1}\bmod{p}$$
for any positive integer $e$.
\end{enumerate}
\end{prop}

Let $\alpha_{\mfM},\alphab_{\mfM}\in\C$ be the roots of
$T^2-a(F_{\mfM})T+\N_{\mfM}$.
Then
$\alpha_{\mfM}+\alphab_{\mfM}=a(F_{\mfM})$
and
$\alpha_{\mfM}\alphab_{\mfM}=\N_{\mfM}$.
We see that the roots of $T^2-a(F_{\mfM}^e)T+(\N_{\mfM})^e$
are $\alpha_{\mfM}^e,\alphab_{\mfM}^e$.
Then $\alpha_{\mfM}^e+\alphab_{\mfM}^e=a(F_{\mfM}^e)$.
We have the following corollary to Proposition \ref{a&rho}(1) (for $e=1$):

\begin{cor}
\label{ainC}

We have $a(F_{\mfM}^e)\in\cC(\N_{\mfM},e)$
for any positive integer $e$.

\end{cor}

For a later use, we give the following lemma:

\begin{lem}
\label{m|a}

Let $m$ be the residue characteristic of $\mfM$.
The the following conditions are equivalent:
\begin{enumerate}[\upshape (i)]
\setlength{\itemsep}{0mm}
\setlength{\parskip}{0mm}
\item
$m\mid a(F_{\mfM})$.
\item
$m\mid a(F_{\mfM}^e)$ for a positive integer $e$.
\item
$m\mid a(F_{\mfM}^e)$ for any positive integer $e$.
\end{enumerate}
\end{lem}

\begin{proof}

For each $e\geq 1$, there is a polynomial $P_e(S,T)\in\Z[S,T]$ such that
$(S+T)^e=S^e+T^e+STP_e(S+T,ST)$.
Then
$a(F_{\mfM})^e=a(F_{\mfM}^e)+\N_{\mfM}P_e(a(F_{\mfM}),\N_{\mfM})$.
Since $m\mid \N_{\mfM}$, we have
$m\mid a(F_{\mfM})$
if and only if
$m\mid a(F_{\mfM}^e)$.

\end{proof}

\section{Proof of the main result}
\label{proof}

Now we prove Theorem \ref{mainthm}.
Suppose that the assumption of Theorem \ref{mainthm} holds.
Assume that there is a point $x\in M^B(k)$.
When $B\otimes_{\Q}k\not\cong\M_2(k)$, 
let $K_0$ be a quadratic extension of $k$ satisfying
$B\otimes_{\Q}K_0\cong\M_2(K_0)$.
Let
\begin{equation*}
K:=
\begin{cases}
k&\text{if $B\otimes_{\Q}k\cong\M_2(k)$,}\\
K_0&\text{if $B\otimes_{\Q}k\not\cong\M_2(k)$.}
\end{cases}
\end{equation*}
Note that the degree $[K:\Q]$ is even.
Then $x$
is represented by a QM-abelian surface $(A,i)$ by $\cO$ over $K$
(see \cite[Theorem 1.1]{J}).
Since $B\not\in\cS(k,\mfq)$, there is a prime divisor $p$ of $d(B)$ such that
$p\ne q$ and
$p$ does not belong to
\begin{equation*}
\begin{cases}
\cP(\cD(\N_{\mfq},e_{\mfq})) & 
\text{if $B\otimes_{\Q}k\cong\M_2(k)$},\\
\cP(\cD(\N_{\mfq},2e_{\mfq})) & 
\text{if $B\otimes_{\Q}k\not\cong\M_2(k)$.}
\end{cases}
\end{equation*}
Fix such $p$, and
let
$$\varrho_p:\G_K\longrightarrow\F_{p^2}^{\times}$$
be a canonical isogeny character at $p$ associated to $(A,i)$. 

By Proposition \ref{r12}, the character
$\varrho_p^{12}$ is unramified outside $p$.
Then it is identified with a character
$\mfI_K(p)\longrightarrow\F_{p^2}^{\times}$,
where $\mfI_K(p)$ is the group of fractional ideals of $K$ prime to $p$.
When $B\otimes_{\Q}k\not\cong\M_2(k)$,
we may assume that $\mfq$ is ramified in $K/k$
by replacing $K_0$ if necessary. 
In any case,
let $\mfQ$ be the unique prime of $K$ above $\mfq$.
Note that $\mfQ$ is the unique prime of $K$ above $q$, and so
$q\cO_K=\mfQ^{e_{\mfQ}}$,
$(\N_{\mfQ})^{e_{\mfQ}}=(q^{f_{\mfQ}})^{e_{\mfQ}}=q^{[K:\Q]}$.
Then by Corollary \ref{rl2}, we have
$$\varrho_p^{12}(F_{\mfQ}^{e_{\mfQ}})
=\varrho_p^{12}(\mfQ^{e_{\mfQ}})
=\varrho_p^{12}(q\cO_K)
=\varrho_p^{12}(1,\cdots,1,q,\cdots,q,\cdots)$$
$$=\varrho_p^{12}(q^{-1},\cdots,q^{-1},1,\cdots,1,\cdots)
=\prod_{\mfP|p}r_p(\mfP)^{12}(q^{-1})
\equiv\prod_{\mfP|p}q^{6e_{\mfP} f_{\mfP}}
=q^{6[K:\Q]}
\bmod{p}.$$
Here,
$(1,\cdots,1,q,\cdots,q,\cdots)$
(\resp $(q^{-1},\cdots,q^{-1},1,\cdots,1,\cdots)$)
is the id\`{e}le of $K$
whose components above $p$ are $1$ and the others $q$
(\resp whose components above $p$ are $q^{-1}$ and the others $1$),
and $\mfP$ runs through the primes of $K$ above $p$.
On the other hand, we have
$$a(F_{\mfQ}^{e_{\mfQ}})\equiv\varrho_p(F_{\mfQ}^{e_{\mfQ}})+(\N_{\mfQ})^{e_{\mfQ}} 
\varrho_p(F_{\mfQ}^{e_{\mfQ}})^{-1}
=\varrho_p(F_{\mfQ}^{e_{\mfQ}})+q^{[K:\Q]} \varrho_p(F_{\mfQ}^{e_{\mfQ}})^{-1}\bmod{p}$$
by Proposition \ref{a&rho}(2).
Let
$\varepsilon:=q^{-\frac{[K:\Q]}{2}} \varrho_p(F_{\mfQ}^{e_{\mfQ}})\in\F_{p^2}^{\times}$.
Then
$$\varepsilon^{12}=1 \qquad\text{and}\qquad
a(F_{\mfQ}^{e_{\mfQ}})\equiv(\varepsilon+\varepsilon^{-1})q^{\frac{[K:\Q]}{2}}\bmod{p}.$$
Therefore
$$a(F_{\mfQ}^{e_{\mfQ}})\equiv 0, \pm q^{\frac{[K:\Q]}{2}}, \pm 2q^{\frac{[K:\Q]}{2}}\bmod{p}
\qquad\text{or}\qquad
a(F_{\mfQ}^{e_{\mfQ}})^2\equiv 3q^{[K:\Q]}\bmod{p}.$$
By Corollary \ref{ainC}, we have $a(F_{\mfQ}^{e_{\mfQ}})\in\cC(\N_{\mfQ},e_{\mfQ})$.
We also have
\begin{equation*}
\N_{\mfQ}=\N_{\mfq}
\qquad\text{and}\qquad
e_{\mfQ}=
\begin{cases}
e_{\mfq}&\text{if $B\otimes_{\Q}k\cong\M_2(k)$,}\\
2e_{\mfq}&\text{if $B\otimes_{\Q}k\not\cong\M_2(k)$.}
\end{cases}
\end{equation*}
Then
$$a(F_{\mfQ}^{e_{\mfQ}}),
a(F_{\mfQ}^{e_{\mfQ}})\pm q^{\frac{[K:\Q]}{2}},
a(F_{\mfQ}^{e_{\mfQ}})\pm 2q^{\frac{[K:\Q]}{2}},
a(F_{\mfQ}^{e_{\mfQ}})^2-3q^{[K:\Q]}
\in\cD(\N_{\mfQ},e_{\mfQ}).$$
Since
$p\not\in\cP(\cD(\N_{\mfq},e_{\mfQ}))$,
we have
\begin{enumerate}
\item
$a(F_{\mfQ}^{e_{\mfQ}})=0, \pm q^{\frac{[K:\Q]}{2}}, \pm 2q^{\frac{[K:\Q]}{2}}$, or
\item
$a(F_{\mfQ}^{e_{\mfQ}})^2=3q^{[K:\Q]}$.
\end{enumerate}

\noindent
[Case (1)]. In this case, we have
$q\mid a(F_{\mfQ}^{e_{\mfQ}})$.
Then by Lemma \ref{m|a}, we have
$q\mid a(F_{\mfQ})$.
Since $f_{\mfQ}(=f_{\mfq})$ is odd, we obtain
$B\otimes_{\Q}\Q(\sqrt{-q})\cong\M_2(\Q(\sqrt{-q}))$
or
($q=2$ and $B\otimes_{\Q}\Q(\sqrt{-1})\cong\M_2(\Q(\sqrt{-1}))$)
(see \cite[Theorem 2.1, Propositions 2.3 and 5.1 (1)]{J}).
This contradicts $B\in\cB(q)$.

\noindent
[Case (2)].
In this case,
$q=3$ and $[K:\Q]$ is odd, which is a contradiction.

Therefore we conclude $M^B(k)=\emptyset$.

\qed

\section{Counterexamples to the Hasse principle}
\label{sec:ex}

\begin{table}[hbtp]
\caption{}
\label{tableABP}
{\small
\begin{center}
\begin{tabular}{c|c|c|c} \hline 
$(N,e)$ & $\cC(N,e)$ & $\cD(N,e)$ & $\cP(\cD(N,e))$ \\ \hline
$(2,2)$ & \parbox{20mm}{\strut{}$0$, $-3$, $-4$\strut} &
\parbox{68mm}{\strut{}$0$, $\pm 1$, $\pm 2$, $-3$, $\pm 4$, $-5$, $-6$, $-7$, $-8$, $-12$\strut} &
\parbox{38mm}{\strut{}$2$, $3$, $5$, $7$\strut} \\ \hline
$(2,4)$ & \parbox{20mm}{\strut{}$1$, $\pm 8$\strut} &
\parbox{68mm}{\strut{}$0$, $1$, $-3$, $\pm 4$, $5$, $-7$, $\pm 8$, $9$, $\pm 12$,
$\pm 16$, $-47$\strut} &
\parbox{38mm}{\strut{}2, 3, 5, 7, 47\strut} \\ \hline
$(2,6)$ & \parbox{20mm}{\strut{}0, 9, $-16$\strut} &
\parbox{68mm}{\strut{}0, 1, $-7$, $\pm$8, 9, $\pm$16, 17, $-24$, 25, $-32$,
64, $-111$, $-192$\strut} &
\parbox{38mm}{\strut{}2, 3, 5, 7, 17, 37\strut} \\ \hline
$(2,8)$ & \parbox{20mm}{\strut{}$-31$, 32\strut} &
\parbox{68mm}{\strut{}0, 1, $-15$, 16, $-31$, 32, $-47$, 48, $-63$, 64, 193, 256\strut} &
\parbox{38mm}{\strut{}2, 3 , 5, 7, 31, 47, 193\strut} \\ \hline
$(2,10)$ & \parbox{20mm}{\strut{}0, 57, $-64$\strut} &
\parbox{68mm}{\strut{}$0$, $-7$, $25$, $\pm 32$, $57$, $\pm 64$, $89$, $-96$,
$121$, $-128$, $177$, $1024$, $-3072$\strut} &
\parbox{38mm}{\strut{}2, 3, 5, 7, 11, 19, 59, 89\strut} \\ \hline
$(2,12)$ & \parbox{20mm}{\strut{}$-47$, $\pm 128$\strut} &
\parbox{68mm}{\strut{}0, 17, $-47$, $\pm 64$, 81, $-111$,
$\pm 128$, $-175$, $\pm 192$, $\pm 256$, 4096, $-10079$\strut} &
\parbox{38mm}{\strut{}2, 3, 5, 7, 17, 37, 47, 10079\strut} \\ \hline
$(2,14)$ & \parbox{20mm}{\strut{}0, $-87$, $-256$\strut} &
\parbox{68mm}{\strut{}0, 41, $-87$, $\pm 128$, 169, $-215$, $\pm 256$, $-343$, $-384$,
$-512$, 16384,$-41583$, $-49152$\strut} &
\parbox{38mm}{\strut{}2, 3, 5, 7, 13, 29, 41, 43, 83, 167\strut} \\ \hline
$(2,16)$ & \parbox{20mm}{\strut{}449, 512\strut} &
\parbox{68mm}{\strut{}0, $-63$, 193, 256, 449, 512, 705, 768, 961, 1024, 4993, 65536\strut} &
\parbox{38mm}{\strut{}2, 3, 5, 7, 31, 47, 193, 449, 4993\strut} \\ \hline
$(3,2)$ & \parbox{20mm}{\strut{}$-2$, 3, $-5$, $-6$\strut} &
\parbox{60mm}{\strut{}0, 1, $-2$, $\pm 3$, 4, $-5$, $\pm 6$, $-8$, $\pm 9$, $-11$,
$-12$, $-18$, $-23$\strut} &
\parbox{38mm}{\strut{}2, 3, 5, 11, 23\strut} \\ \hline
$(3,4)$ & \parbox{20mm}{\strut{}7, $-9$, $-14$, 18\strut} &
\parbox{68mm}{\strut{}0, $-2$, 4, $-5$, 7, $\pm 9$, $-11$, $-14$, 16,
$\pm 18$, $-23$, 25, $\pm 27$, $-32$, 36, $-47$, 81, $-162$, $-194$\strut} &
\parbox{38mm}{\strut{}2, 3, 5, 7, 11, 23, 47, 97\strut} \\ \hline
$(3,6)$ & \parbox{20mm}{\strut{}10, 46, $-54$\strut} &
\parbox{68mm}{\strut{}0, $-8$, 10, $-17$, 19, $-27$, 37, $-44$, 46, $-54$, 64,
$-71$, 73, $-81$, 100, $-108$, 729, $-2087$\strut} &
\parbox{38mm}{\strut{}2, 3, 5, 11, 17, 19, 23, 37, 71, 73, 2087\strut} \\ \hline
$(3,8)$ & \parbox{20mm}{\strut{}34, $-81$, $-113$, 162\strut} &
\parbox{68mm}{\strut{}0, $-32$, 34, $-47$, 49, $\pm 81$, $-113$, 115, $-128$, $\pm 162$,
$-194$, 196, $\pm 243$, $-275$, 324, 6561, $-6914$, $-13122$, $-18527$\strut} &
\parbox{38mm}{\strut{}2, 3, 5, 7, 11, 17, 23, 47, 97, 113, 191, 3457\strut} \\ \hline
$(3,10)$ & \parbox{20mm}{\strut{}243, 475, $-482$, $-486$\strut} &
\parbox{68mm}{\strut{}0, 4, $-11$, 232, $-239$, $\pm 243$, 475, $-482$, $\pm 486$, 718,
$-725$, $\pm 729$, 961, $-968$, $-972$, 48478, 55177, 59049, $-118098$\strut} &
\parbox{38mm}{\strut{}2, 3, 5, 11, 19, 23, 29, 31, 239, 241, 359, 2399, 24239\strut} \\ \hline
$(3,12)$ & \parbox{20mm}{\strut{}658, $-1358$, 1458\strut} &
\parbox{68mm}{\strut{}0, $-71$, 100, $-629$, 658, 729, $-800$, $-1358$, 1387, 1458,
$-2087$, 2116, 2187, $-2816$, 2916, 249841, 531441, $-1161359$\strut} &
\parbox{38mm}{\strut{}2, 3, 5, 7, 11, 17, 19, 23, 37, 47, 71, 73, 97, 433,
577, 1009, 1151, 2087\strut} \\ \hline
$(3,14)$ & \parbox{20mm}{\strut{}2187, 2515, 3022, $-4374$\strut} &
\parbox{60mm}{\strut{}0, 328, 835, $-1352$, $-1859$, $\pm 2187$, 2515, 3022,
$\pm 4374$, 4702, 5209, $\pm 6561$, 6889, 7396, $-8748$, 4782969, $-5216423$,
$-8023682$, $-9565938$\strut} &
\parbox{38mm}{\strut{}2, 3, 5, 11, 13, 23, 41, 43, 83, 167, 337, 503, 673, 1511,
2351, 5209, 24023\strut} \\ \hline
$(3,16)$ & \parbox{20mm}{\strut{}$-353$, $-6561$, $-11966$, 13122\strut} &
\parbox{68mm}{\strut{}0, $-353$, 1156, $-5405$, 6208, $\pm 6561$, $-6914$,
$-11966$, 12769, $\pm 13122$, $-13475$, $-18527$, $\pm 19683$, $-25088$, 26244, 14044993,
43046721, $-86093442$, $-129015554$\strut} &
\parbox{38mm}{\strut{}2, 3, 5, 7, 11, 17, 23, 31, 47, 97, 113, 191, 193, 353,
383, 2113, 3457, 30529, 36671\strut} \\ \hline
\end{tabular}
\end{center}
}
\end{table}

We have computed the sets $\cC(N,e),\cD(N,e),\cP(\cD(N,e))$
in several cases
as seen in Table \ref{tableABP}.
Then we obtain the following counterexamples to the Hasse principle
on $M^B$
over number fields:

\begin{prop}
\label{prop:ex}
\begin{enumerate}[\upshape (1)]
\item
Let $d(B)=39$, and let
$k=\Q(\sqrt{2},\sqrt{-13})$ or $\Q(\sqrt{-2},\sqrt{-13})$.
Then
$B\otimes_{\Q}k\cong\M_2(k)$,
$M^B(k)=\emptyset$
and
$M^B(k_v)\ne\emptyset$
for any place $v$ of $k$.
Here, $k_v$ is the completion of $k$ at $v$.
\item
Let $L$ be the subfield of $\Q(\zeta_9)$
satisfying $[L:\Q]=3$,
and
let $(d(B),k)=(62,L(\sqrt{-39}))$ or $(86,L(\sqrt{-15}))$.
Then
$B\otimes_{\Q}k\not\cong\M_2(k)$,
$M^B(k)=\emptyset$
and
$M^B(k_v)\ne\emptyset$
for any place $v$ of $k$.
\end{enumerate}
\end{prop}

\begin{proof}

(1)
The prime number $3$ (\resp $13$) is inert (\resp ramified) in $\Q(\sqrt{-13})$.
Then
$B\otimes_{\Q}\Q(\sqrt{-13})\cong\M_2(\Q(\sqrt{-13}))$,
and so
$B\otimes_{\Q}k\cong\M_2(k)$.

Applying Theorem \ref{mainthm} to $q=2$, we obtain
$M^B(k)=\emptyset$.
In fact,
$(e_{\mfq},f_{\mfq})=(4,1)$
where $\mfq$ is the unique prime of $k$ above $q=2$,
and the prime divisor $13$ of $d(B)$ does not belong to
$\cP(\cD(2,4))\cup\{2\}$ (see Table \ref{tableABP}).
Since $3$ (\resp $13$) splits in $\Q(\sqrt{-2})$ (\resp $\Q(\sqrt{-1})$),
we have
$B\otimes_{\Q}\Q(\sqrt{-2})\not\cong\M_2(\Q(\sqrt{-2}))$
(\resp $B\otimes_{\Q}\Q(\sqrt{-1})\not\cong\M_2(\Q(\sqrt{-1}))$).

By \cite[p.94]{J}, we have
$M^B(\Q(\sqrt{-13})_w)\ne\emptyset$
for any place $w$ of $\Q(\sqrt{-13})$
(\cf\cite{JL}).
Therefore $M^B(k_v)\ne\emptyset$ for any place $v$ of $k$.

(2)
For a field $F$ of characteristic $\ne 2$ and two elements $a,b\in F^{\times}$,
let
$$\left(\frac{a,b}{F}\right)=F+Fe+Ff+Fef$$
be the quaternion algebra over $F$ defined by
$$e^2=a,\ f^2=b,\ ef=-fe.$$
For a prime number $p$, let $e_p,f_p,g_p$
be the ramification index of $p$ in $k/\Q$,
the degree of the residue field extension above $p$ in $k/\Q$,
and the number of primes of $k$ above $p$ respectively.

Let $(d(B),k)=(62,L(\sqrt{-39}))$ (\resp $(86,L(\sqrt{-15}))$).
First, we prove $B\otimes_{\Q}k\not\cong\M_2(k)$.
We see
$\displaystyle B\cong\left(\frac{62,13}{\Q}\right)$
(\resp $\displaystyle \left(\frac{86,5}{\Q}\right)$)
by \cite[\S 3.6 g)]{Shimizu}.
We have
$(e_2,f_2,g_2)=(1,3,2)$.
Let $v$ be place of $k$ above $2$.
By the same argument as in the proof of \cite[Proposition 8.1]{A3},
we have $B\otimes_{\Q}k_v\not\cong\M_2(k_v)$.
Therefore $B\otimes_{\Q}k\not\cong\M_2(k)$.

Applying Theorem \ref{mainthm} to $q=3$, we obtain
$M^B(k)=\emptyset$.
In fact, $(e_{\mfq},f_{\mfq})=(6,1)$
where $\mfq$ is the unique prime of $k$ above $q=3$,
and the prime divisor $31$ (\resp $43$) of $d(B)$ does not belong to
$\cP(\cD(3,12))\cup\{3\}$.
Since $31$ (\resp $43$) splits in $\Q(\sqrt{-3})$,
we have
$B\otimes_{\Q}\Q(\sqrt{-3})\not\cong\M_2(\Q(\sqrt{-3}))$.

By \cite[Table 1]{RdV}, we have
$M^B(\Q(\sqrt{-39})_w)\ne\emptyset$
(\resp $M^B(\Q(\sqrt{-15})_w)\ne\emptyset$)
for any place $w$ of $\Q(\sqrt{-39})$ (\resp $\Q(\sqrt{-15})$).
Therefore $M^B(k_v)\ne\emptyset$ for any place $v$ of $k$.

\end{proof}

\def\bibname{References}

(Keisuke Arai)
Department of Mathematics, School of Engineering,
Tokyo Denki University,
5 Senju Asahi-cho, Adachi-ku, Tokyo 120-8551, Japan

\textit{E-mail address}: \texttt{araik@mail.dendai.ac.jp}

\end{document}